\documentclass[11pt]{amsart}

\usepackage{color}
\usepackage[colorlinks=true,citecolor=citations,linkcolor=numbering,urlcolor=email,backref=page]{hyperref}
\definecolor{numbering}{rgb}{0.7,0,0}
\definecolor{email}{RGB}{073,103,141}
\definecolor{citations}{RGB}{034,113,179}

\newcommand{\eps}{\varepsilon}
\renewcommand{\Re}{\mathrm{Re}}
\renewcommand{\Im}{\mathrm{Im}}
\newcommand{\pd}[2]{\frac{\partial #1}{\partial #2}}
\renewcommand{\d}{\mathrm{d}}
\renewcommand{\i}{\mathrm{i}}
\newcommand{\R}{\mathbb{R}}
\newcommand{\C}{\mathbb{C}}
\newcommand{\ov}{\overline}
\renewcommand{\j}{\bar\jmath}

\numberwithin{equation}{section}

\newtheorem{theorem}{Theorem}[section]
\newtheorem{proposition}{Proposition}[section]
\newtheorem{corollary}{Corollary}[section]
\theoremstyle{remark}
\newtheorem{remark}{Remark}[section]
\newtheorem{example}{Example}[section]

\date{April 2, 2014}

\thanks{The author was supported by Forschungsinstitut f\"ur Mathematik (FIM) at ETH Z\"urich and
Schweizerischer Nationalfonds SNF via the postdoctoral fellowship PA00P2\_142053. The author is grateful to Friedrich-Schiller-Universit\"at Jena for financial support for several trips to Jena where a part of the writing for this article took place.}

\author[T.~Mettler]{Thomas Mettler}
\address{Department of Mathematics, ETH Z\"urich, Switzerland}
\email{thomas.mettler@math.ethz.ch}
\title[On K\"ahler metrisability]{On K\"ahler metrisability of two-dimensional complex projective structures}

\begin{document}
\begin{abstract}
We derive necessary conditions for a com\-plex pro\-jec\-tive structure on a complex surface to arise via the Levi-Civita connection of a (pseudo-)K\"ahler metric. Furthermore we show that the (pseudo-)K\"ahler metrics defined on some domain in the projective plane which are compatible with the standard com\-plex pro\-jec\-tive structure are in one-to-one correspondence with the hermitian forms on $\C^3$ whose rank is at least two. This is achieved by prolonging the relevant finite-type first order linear differential system to closed form. Along the way we derive the com\-plex pro\-jec\-tive Weyl and Liouville curvature using the language of Cartan geometries.
\end{abstract}
\maketitle

\section{Introduction}
Recall that an equivalence class of affine torsion-free connections on the tangent bundle of a smooth manifold $N$ is called a (real) projective structure~\cite{MR1504846,51.0569.03,48.0844.04}. Two connections $\nabla$ and $\nabla^{\prime}$ are \textit{projectively equivalent} if they share the same unparametrised geodesics. This condition is equivalent to $\nabla$ and $\nabla^{\prime}$ inducing the same parallel transport on the projectivised tangent bundle $\mathbb{P}TN$. 
 
It is a natural task to (locally) characterise the projective structures arising via the Levi-Civita connection of a (pseudo-)Riemannian metric. R. Liouville~\cite{21.0317.01} made the crucial observation that the Riemannian metrics on a surface whose Levi-Civita connection belongs to a given projective class precisely correspond to nondegenerate solutions of a certain projectively invariant finite-type linear system of partial differential equations. In~\cite{MR2581355} Bryant, Eastwood and Dunajski used Liouville's observation to solve the two-dimensional version of the Riemannian metrisability problem. In another direction it was shown in~\cite{MR3144212} that on a surface locally every affine torsion-free connection is projectively equivalent to a conformal connection (see also~\cite{MR3043749}). Local existence of a connection with skew-symmetric Ricci tensor in a given projective class was investigated in~\cite{MR3144367} (see also~\cite{arXiv:1303.4912} for a connection to Veronese webs). Liouville's result generalises to higher dimensions~\cite{MR1384327} and the corresponding finite-type differential system was prolonged to closed form in~\cite{MR2384718,MR1384327}. Several necessary conditions for Riemann metrisability of a projective structure in dimensions larger than two were given in~\cite{MR2876789}. See also~\cite{MR3158041,MR3159950} for the role of Einstein metrics in projective geometry.

Now let $M$ be a complex manifold of complex dimension $d>1$ with integrable almost complex structure map $J$. Two affine torsion-free connections $\nabla$ and $\nabla^{\prime}$ on $TM$ which preserve $J$ are called \textit{com\-plex pro\-jec\-tively equivalent} if they share the same \textit{generalised geodesics} (for the notion of a curved complex projective structure on Riemann surfaces see~\cite{MR1656822}). A \textit{generalised geodesic} is a smoothly immersed curve $\gamma \subset M$ with the property that the $2$-plane spanned by $\dot{\gamma}$ and $J \dot\gamma$ is parallel along $\gamma$. Complex projective geometry was introduced by Otsuki and Tashiro~\cite{MR0066024,MR0087181}. Background on the history of complex projective geometry and its recently discovered connection to Hamiltonian $2$-forms (see~\cite{MR2228318} and references therein) may be found in~\cite{MR2998672}. 

In the complex setting it is natural to study the \textit{K\"ahler metrisability problem}, i.e.~try to (locally) characterise the com\-plex pro\-jec\-tive structures which arise via the Levi-Civita connection of a (pseudo-)K\"ahler metric. Similar to the real case, the K\"ahler metrics whose Levi-Civita connection belongs to a given complex projective class precisely correspond to nondegenerate solutions of a certain complex projectively invariant finite-type linear system of partial differential equations~\cite{MR492674,MR2998672,MR1619720}. 

In this note we prolong the relevant differential system to closed form in the surface case. In doing so we obtain necessary conditions for K\"ahler metrisability of a com\-plex pro\-jec\-tive structure $[\nabla]$ on a complex surface and show in particular that the generic com\-plex pro\-jec\-tive structure is not K\"ahler metrisable. Furthermore we show that the space of K\"ahler metrics compatible with a given com\-plex pro\-jec\-tive structure is algebraically constrained by the com\-plex pro\-jec\-tive Weyl curvature of $[\nabla]$. We also show that the (pseudo-)K\"ahler metrics defined on some domain in $\mathbb{CP}^2$ which are compatible with the standard com\-plex pro\-jec\-tive structure are in one-to-one correspondence with the hermitian forms on $\C^3$ whose rank is at least two. A result whose real counterpart is a well-known classical fact. This note concerns itself with the complex 2-dimensional case, but there are obvious higher dimensional generalisations that can be treated with the same techniques.  

The reader should be aware that the results presented here can also be obtained by using the elegant and powerful theory of Bernstein--Gelfand--Gelfand (BGG) sequences developed by \v{C}ap, Slov\'ak and Sou\v{c}ek~\cite{MR1847589} (see also the article of Calderbank and Diemer~\cite{MR1856258}). In particular, the prolongation computed here is an example of a prolongation connection of a first BGG equation in parabolic geometry and may be derived using the techniques developed in~\cite{MR2984590}.  

This note aims at providing an intermediate analysis between the abstract BGG machinery and pure local coordinate computations. This is achieved by carrying out the computations on the parabolic Cartan geometry of a com\-plex pro\-jec\-tive surface.

\section{Com\-plex pro\-jec\-tive surfaces}

\subsection{Definitions}

Let $M$ be a complex $2$-manifold with integrable almost complex structure map $J$ and $\nabla$ an affine torsion-free connection on $TM$. We call $\nabla$ \textit{complex-linear} if $\nabla J=0$. A \textit{generalised geodesic} for $\nabla$ is a smoothly immersed curve $\gamma \subset M$ with the property that the $2$-plane spanned by $\dot{\gamma}$ and $J \dot\gamma$ is parallel along $\gamma$, i.e.~$\gamma$ satisfies the reparametrisation invariant condition
\begin{equation}\label{hplanarcurve}
\nabla_{\dot\gamma}\dot\gamma\wedge \dot\gamma\wedge J\dot\gamma=0.
\end{equation}
We call two complex linear torsion-free connections $\nabla$ and $\nabla^{\prime}$ on $M$ \textit{com\-plex pro\-jec\-tively equivalent}, if they have the same generalised geodesics. An equivalence class of com\-plex pro\-jec\-tively equivalent connections is called a \textit{com\-plex pro\-jec\-tive structure} and will be denoted by $[\nabla]$. A complex $2$-manifold equipped with a com\-plex pro\-jec\-tive structure will be called a \textit{com\-plex pro\-jec\-tive surface}. 

\begin{remark}
What we here call a complex projective structure was originally called a \textit{holomorphic projective structure} by Tashiro~\cite{MR0087181} and others. Once it was realised that in general complex projective structures are not holomorphic in any reasonable way, the name \textit{h-projective structure} was used -- and is still so -- see for instance~\cite{MR2948791,MR2729086,MR2998672}. Furthermore, what we here call generalised geodesics are called \textit{h-planar curves} in the literature using the name h-projective. One might argue that the notion of a complex projective structure can be confused with well-established notions in algebraic geometry. For this reason complex projective is sometimes also abbreviated to c-projective (see for instance~\cite{arxiv:1311.0517}).  
\end{remark}

Extending $\nabla$ to the complexified tangent bundle $T^{\C}M \to M$, it follows from the complex linearity of $\nabla$ that for every local holomorphic coordinate system $z=(z^i) : U \to \C^2$ on $M$ there exist unique complex-valued functions $\Gamma^i_{jk}$ on $U$, so that
$$
\nabla_{\partial_{z^j}}\partial_{z^k}=\Gamma^i_{jk}\partial_{z^i}.
$$
We call the functions $\Gamma^i_{jk}$ the \textit{complex Christoffel symbols} of $\nabla$. Tashiro showed~\cite{MR0087181} that two torsion-free complex linear connections $\nabla$ and $\nabla^{\prime}$ on $M$ are com\-plex pro\-jec\-tively equivalent if and only if there exists a $(1,\! 0)$-form $\beta \in \Omega^{1,0}(M,\R)$ so that
\begin{equation}\label{hprojequiv}
\nabla^{\prime}_ZW-\nabla_ZW=\beta(Z)W+\beta(W)Z
\end{equation}
for all $(1,\! 0)$ vector fields $Z,W \in \Gamma(T^{1,0}M)$. In analogy to the real case one can use \eqref{hprojequiv} to show that $\nabla$ and $\nabla^{\prime}$ are complex projectively equivalent if and only if they induce the same parallel transport on the complex projectivised tangent bundle $\mathbb{P}T^{1,0}M$. 

Writing $\Gamma^i_{jk}$ and $\hat{\Gamma}^i_{jk}$ for the complex Christoffel symbols of $\nabla$ and $\nabla^{\prime}$ with respect to some holomorphic coordinates $z=(z^i)$ and $\beta=\beta_i \d z^i$, equation \eqref{hprojequiv} translates to
\begin{equation}\label{hprojequivcoord}
\hat{\Gamma}^i_{jk}=\Gamma^i_{jk}+\delta^i_j\beta_k+\delta^i_k\beta_j. 
\end{equation}
Note that formally equation \eqref{hprojequivcoord} is identical to the equation relating two real projectively equivalent connections on a real manifold. In particular, similarly to the real case (see~\cite{MR1504846,51.0569.03}), the functions
\begin{equation}\label{cplxprojinv}
\Pi^i_{jk}=\Gamma^i_{jk}-\frac{1}{3}\left(\Gamma^l_{lj}\delta^i_k+\Gamma^l_{lk}\delta^i_j\right)
\end{equation}
are com\-plex pro\-jec\-tively invariant in the sense that they only depend on the coordinates $z$. Moreover locally $[\nabla]$ can be recovered from the functions $\Pi^i_{jk}$ and two torsion-free complex linear connections are com\-plex pro\-jec\-tively equivalent if and only if they give rise to the same functions $\Pi^i_{jk}$ in every holomorphic coordinate system. 

A com\-plex pro\-jec\-tive structure $[\nabla]$ is called \textit{holomorphic} if the functions $\Pi^i_{jk}$ are holomorphic in every holomorphic coordinate system. Gunning~\cite{MR505691} obtained relations on characteristic classes of complex manifolds carrying holomorphic projective structures. The condition on a manifold to carry a holomorphic projective structure is particularly restrictive in the case of compact complex surfaces. See also the beautiful twistorial interpretation of holomorphic projective surfaces by Hitchin~\cite{MR699802} and Remark~\ref{remcplxproj}.  

\subsection{Cartan geometry}
A com\-plex pro\-jec\-tive structure admits a description in terms of a \textit{normal Cartan geometry} modelled on complex projective space $\mathbb{CP}^n$, following the work of Ochiai~\cite{MR0284936}: see~~\cite{MR2591680} and~\cite{MR0500679}. The reader unfamiliar with Cartan geometries may consult~\cite{MR2532439} for a modern introduction. We will restrict to the construction in the complex two-dimensional case. 

Let $\mathrm{PSL}(3,\C)$ act on $\mathbb{CP}^2$ from the left in the obvious way and let $P$ denote the stabiliser subgroup of the element $[1,0,0]^t \in \mathbb{CP}^2$. We have:
\begin{theorem}\label{cartangeom}
Suppose $(M,J,[\nabla])$ is a com\-plex pro\-jec\-tive surface. Then there exists (up to isomorphism) a unique real Cartan geometry $(\pi : B \to M,\theta)$ of type $(\mathrm{PSL}(3,\C),P)$ such that for every local holomorphic coordinate system $z=(z^i) : U \to \C^2$, there exists a unique section $\sigma_z : U \to B$ satisfying
\begin{equation}\label{pullbackconn}
(\sigma_z)^*\theta=\left(\begin{array}{ccc} 0 & \phi^0_1 & \phi^0_2\\ \phi^1_0 & \phi^1_1 & \phi^1_2 \\ \phi^2_0 & \phi^2_1 & \phi^2_2\end{array}\right)
\end{equation}
where $$
\phi^i_0=\d z^i, \quad \text{and} \quad \phi^i_j=\Pi^i_{jk}\d z^k, \quad \text{and} \quad \phi^0_i=\Pi_{ik}\d z^k,  
$$
with
$$
\Pi_{ij}=\Pi^k_{il}\Pi^l_{jk}-\frac{\partial \Pi^k_{ij}}{\partial z^k}
$$
and $\Pi^i_{jk}$ denote the com\-plex pro\-jec\-tive invariants with respect to $z^i$ defined in \eqref{cplxprojinv}. 
\end{theorem}

\begin{remark}
Suppose $\varphi : (M,J,[\nabla]) \to (M^{\prime},J^{\prime},[\nabla]^{\prime})$ is a biholomorphism between complex projective surfaces identifying the complex projective structures, then there exists a diffeomorphism $\hat{\varphi} : B \to B^{\prime}$ which is a $P$-bundle map covering $\varphi$ and which satisfies $\hat{\varphi}^*\theta^{\prime}=\theta$. Conversely, every diffeomorphism $\Phi : B \to B^{\prime}$ that is a $P$-bundle map and satisfies $\Phi^*\theta^{\prime}=\theta$ is of the form $\Phi=\hat{\varphi}$ for a unique biholomorphism $\varphi : M \to M^{\prime}$ identifying the complex projective structures. 
\end{remark}

\begin{example}
Let $B=\mathrm{PSL}(3,\C)$ and let $\theta$ denote its Maurer-Cartan form. Setting $M=B/P\simeq \mathbb{CP}^2$ and $\pi : \mathrm{PSL}(3,\C) \to \mathbb{CP}^2$ the natural quotient projection, one obtains a complex projective structure on $\mathbb{CP}^2$ whose generalised geodesics are the smoothly immersed curves $\gamma \subset \mathbb{CP}^1$ where $\mathbb{CP}^1\subset \mathbb{CP}^2$ is any linearly embedded projective line. This is precisely the complex projective structure associated to the Levi-Civita connection of the Fubini-Study metric on $\mathbb{CP}^2$. This example satisfies $\d\theta+\theta\wedge\theta=0$ and is hence called \textit{flat}.  
\end{example}

Let $(\pi : B \to M,\theta)$ be the Cartan geometry of a complex projective structure $(J,[\nabla])$ on a simply-connected surface $M$ whose Cartan connection satisfies $\d \theta+\theta\wedge\theta=0$. Then there exists a local diffeomorphism $\Phi : B \to \mathrm{PSL}(3,\C)$ pulling back the Maurer-Cartan form of $\mathrm{PSL}(3,\C)$ to $\theta$ and consequently, a local biholomorphism $\varphi : M \to \mathbb{CP}^2$ identifying the projective structure on $M$ with the standard flat structure on $\mathbb{CP}^2$. 
\subsection{Bianchi-identities}
Theorem \ref{cartangeom} implies that the curvature form $\Theta=\d\theta+\theta\wedge\theta$ satisfies
\begin{equation}\label{firststruc}
\Theta=\d\theta+\theta\wedge\theta=\left(\begin{array}{ccc} 0 & \Theta^0_1 & \Theta^0_2\\ 0 & \Theta^1_1 & \Theta^1_2\\ 0 & \Theta^2_1 & \Theta^2_2\end{array}\right)
\end{equation}
with
$$
\Theta^0_i=L_i\theta^1_0\wedge\theta^2_0+K_{il\j}\theta^l_0\wedge\ov{\theta^{\jmath}_0}, \quad \Theta^i_k=W^i_{kl\j}\theta^l_0\wedge\ov{\theta^\jmath_0}
$$
for unique complex-valued functions $L_i,K_{il\j}$, and $W^i_{kl\j}$ on $B$ satisfying $W^l_{li\bar \jmath}=0$. Note that by construction, with respect to local holomorphic coordinates $z=(z^i)$, we obtain
\begin{equation}\label{coordpb}
(\sigma_z)^*W^i_{kl\j}=-\pd{\Pi^i_{kl}}{\bar z^j}.
\end{equation}

 Differentiation of the structure equations \eqref{firststruc} gives
$$
0=\d^2\theta^i_0=W^i_{lk\j}\theta^l_0\wedge\theta^k_0\wedge\ov{\theta^\jmath_0}, \quad \text{and} \quad 0=\d^2 \theta^0_0=K_{ik\j}\theta^i_0\wedge\theta^k_0\wedge\ov{\theta^\jmath_0}
$$
which yields the algebraic Bianchi-identities 
$$
W^i_{lk\j}=W^i_{kl\j}, \quad \text{and} \quad K_{ik\j }=K_{ki\j}.
$$

\subsubsection{Com\-plex pro\-jec\-tive Weyl curvature}

The identities $\d^2\theta^i_k=0$ yield 
$$
\kappa^i_{kl\j}\wedge\theta^l_0\wedge\ov{\theta^{\jmath}_0}=0
$$ with
$$
\kappa^i_{kl\j}=\d W^i_{kl\j}+W^i_{kl\j}\left(\theta^0_0+\ov{\theta^0_0}\right)+K_{kl\j}\theta^i_0-W^i_{ls\j}\theta^s_k-W^i_{ks\j}\theta^s_l+W^s_{kl\j}\theta^i_s-W^i_{kl\bar s}\ov{\theta^s_l}
$$
which implies that there exist complex-valued functions $W^i_{kl\j\bar s}$ and $W^i_{kl\j s}$ on $B$ satisfying 
$$
W^i_{kl\j\bar s}=W^i_{lk\j\bar s}=W^i_{kl\bar s\j}, \quad W^k_{kl\j\bar s}=W^k_{kl\j s}=0, \quad W^i_{kl\j s}=W^i_{lk\j s}
$$
such that
\begin{equation}\label{diffbianchi}
\d W^i_{kl\j}=\left(W^i_{kl\j s}+\delta^i_kK_{sl\j}+\delta^i_lK_{sk\j}-3\delta^i_sK_{kl\j}\right)\theta^s_0+W^i_{kl\j \bar s}\ov{\theta^s_0}+\varphi^i_{kl\j}
\end{equation}
where 
\begin{equation}\label{diffbianchi1a}
\varphi^i_{kl\j}=-W^i_{kl\j}\left(\theta^0_0+\ov{\theta^0_0}\right)+W^i_{ls\j}\theta^s_k+W^i_{ks\j}\theta^s_l-W^s_{kl\j}\theta^i_s+W^i_{kl\bar s}\ov{\theta^s_j}.
\end{equation}

Let $\mathrm{End}_0(TM,J)$ denote the bundle whose fibre at $p \in M$ consists of the $J$-linear endomorphisms of $T_pM$ which are complex-traceless. It follows with the structure equations (\ref{firststruc},\ref{diffbianchi},\ref{diffbianchi1a}) and straightforward computations, that there exists a unique $(1,\!1)$-form $W$ on $M$ with values in $\mathrm{End}_0(TM,J)$ for which 
$$
W\left(\pd{}{z^l},\pd{}{\ov{z^{\jmath}}}\right)\pd{}{z^k}=(\sigma_z)^*W^i_{kl\j}\pd{}{z^i}=-\pd{\Pi^i_{kl}}{\bar z^j}\pd{}{z^i}
$$
in every local holomorphic coordinate system $z=(z^i)$ on $M$. Here, as usual, we extend tensor fields on $M$ complex multilinearly to the complexified tangent bundle of $M$. The bundle-valued $2$-form $W$ is called the \textit{com\-plex pro\-jec\-tive Weyl curvature} of $[\nabla]$.  
We obtain:
\begin{proposition}
A com\-plex pro\-jec\-tive structure $[\nabla]$ on a complex surface $(M,J)$ is holomorphic if and only if the com\-plex pro\-jec\-tive Weyl tensor of $[\nabla]$ vanishes.  
\end{proposition}

\subsubsection{Com\-plex pro\-jec\-tive Liouville curvature}

From $\d^2\theta^0_i\wedge\ov{\theta^1_0}\wedge\ov{\theta^2_0}=0$ one sees after a short computation that
\begin{equation}\label{diffbianchi3}
\d L_i=-4L_i\theta^0_0+L_j\theta^j_i+L_{ij}\theta^j_0+L_{i\j}\ov{\theta^{\jmath}_0}
\end{equation}
for unique complex-valued functions $L_{i\j},L_{ij}$ on $B$. Using this last equation it is easy to check that the $\pi$-semibasic quantity
\begin{equation}\label{liouvillecurv}
(L_1\theta^1_0+L_2\theta^2_0)\otimes \left(\theta^1_0\otimes \theta^2_0\right)
\end{equation}
is invariant under the $P$ right action and thus the $\pi$-pullback of a tensor field $\lambda$ on $M$ which is called the \textit{com\-plex pro\-jec\-tive Liouville curvature} (see the note of R.~Liouville~\cite{19.0317.02} for the construction of $\lambda$ in the real case). 
\begin{remark}
In the case of real projective structures on surfaces, the projective Weyl curvature vanishes identically. Furthermore, note that contrary to the complex projective Liouville curvature, the com\-plex pro\-jec\-tive Weyl tensor exists as well in higher dimensions, but also contains $(2,\! 0)$ parts (see~\cite{MR0087181} for details).  
\end{remark}

The differential Bianchi-identity \eqref{diffbianchi} implies that if the functions $W^i_{kl\j}$ vanish identically, then the functions $K_{ik\j}$ must vanish identically as well. We have thus shown: 
\begin{proposition}
A com\-plex pro\-jec\-tive structure $[\nabla]$ on a complex surface $(M,J)$ is flat if and only the com\-plex pro\-jec\-tive Liouville and Weyl curvature vanish.
\end{proposition}
\begin{remark}\label{remcplxproj}
In~\cite{MR575449} Kobayashi and Ochiai classified compact complex surfaces carrying flat com\-plex pro\-jec\-tive structures. More recently Dumitrescu \cite{MR2574978} showed among other things that a holomorphic projective structure on a compact complex surface must be flat (see also the results by McKay about holomorphic Cartan geometries~\cite{MR2770434}). 
\end{remark} 
\subsubsection{Further identities}

We also obtain
$$
0=\d^2\theta^0_i=\kappa_{ik\j}\wedge\ov{\theta^{\jmath}_0}\wedge\theta^k_0 
$$
with\footnote{We write $\eps_{ij}$ for the antisymmetric $2$-by-$2$ matrix satisfying $\eps_{12}=1$ and $\eps^{ij}$ for the inverse matrix.}
$$
\aligned
\kappa_{ik\j}=&-\d K_{ik\j}+\frac{1}{2}\eps_{sk}L_{i\j}\theta^s_0-K_{ik\j}\left(2\theta^0_0+\ov{\theta^0_0}\right)+K_{sk\j}\theta^s_i+K_{si\j}\theta^s_k-\\
&-W^s_{ik\j}\theta^0_s+K_{ik\bar s}\ov{\theta^s_j}.
\endaligned
$$
It follows that there are complex-valued functions $K_{ik\j l}$ and $K_{kl\bar\imath\j}$ on $B$ satisfying
$$
K_{ik\j l}=K_{ki \j l}, \quad \text{and} \quad K_{kl\bar \imath \j}=K_{lk\bar \imath \j}=K_{kl \j\bar \imath}
$$
such that
\begin{equation}\label{diffbianchi2}
\d K_{ik\j}=\left(K_{ik\j s}+\frac{1}{4}\left(\eps_{sk}L_{i\j}+\eps_{si}L_{k\j}\right)\right)\theta^s_0+K_{ik\j\bar s}\ov{\theta^s_0}+\varphi_{ik\j}
\end{equation}
where
$$
\varphi_{ik\j}=-K_{ik\j}\left(2\theta^0_0+\ov{\theta^0_0}\right)+K_{sk\j}\theta^s_{i}+K_{si\j}\theta^s_k-W^s_{ik\j}\theta^0_s+K_{ik\bar s}\ov{\theta^s_j}.
$$
\subsection{Complex and generalised geodesics}

It is worth explaining how the generalised geodesics of $[\nabla]$ appear in the Cartan geometry $(\pi : B \to M,\theta)$. To this end let $G\subset P\subset \mathrm{PSL}(3,\C)$ denote the quotient group of the group of upper triangular matrices of unit determinant modulo its center. The quotient $B/G$ is the total space of a fibre bundle over $M$ whose fibre $P/G$ is diffeomorphic to $\mathbb{CP}^1$. In fact, $B/G$ may be identified with the total space of the the com\-plex pro\-jec\-tivised tangent bundle $\tau : \mathbb{P}(T^{1,0}M) \to M$ of $(M,J)$. Writing $\theta=(\theta^i_j)_{i,j=0..2}$,
Theorem \ref{cartangeom} implies that the real codimension $4$-subbundle of $TB$ defined by $\theta^2_0=\theta^2_1=0$ descends to a real rank $2$ subbundle $E\subset T\mathbb{P}(T^{1,0}M)$. The integral manifolds of $E$ can most conveniently be identified in local coordinates. Let $z=(z^1,z^2) : U \to \C^2$ be a local holomorphic coordinate system on $M$ and write $\phi$ for the pullback of $\theta$ with the unique section $\sigma_z$ associated to $z$ in Theorem~\ref{cartangeom}. We obtain a local trivialisation of Cartan's bundle
$$
\varphi : U \times P \to \pi^{-1}(U) 
$$
so that for $(z,p) \in U\times P$ we have
\begin{equation}\label{loctriv}
(\varphi^*\theta)_{(z,p)}=\left(\omega_P\right)_p+\mathrm{Ad}(p^{-1})\circ \phi_z
\end{equation}
where $\omega_P$ denotes the Maurer-Cartan form of $P$ and $\mathrm{Ad}$ the adjoint representation of $\mathrm{PSL}(3,\C)$. Consider the Lie group $\tilde{P}\subset \mathrm{SL}(3,\C)$ whose elements are of the form
\begin{equation}\label{covergroup}
\left(\begin{array}{cc} \det a^{-1} & b \\ 0 & a \end{array}\right)
\end{equation}
for $ a\in \mathrm{GL}(2,\C)$ and $b^{t} \in \C^2$. By construction, the elements of $P$ are equivalence classes of elements in $\tilde{P}$ where two elements are equivalent if they differ by scalar multiplication with a complex cube root of $1$. The canonical projection $\tilde{P} \to P$ will be denoted by $\nu$. Note that a piece $N$ of an integral manifold of $E$ that is contained in $\tau^{-1}(U)$ is covered by a map 
$$
(z^1,z^2,p) : N \to U \times \tilde{P}
$$
where $p : N \to \tilde{P}$ may be taken to be of the form
$$
p=\left(\begin{array}{ccr} \frac{1}{(a_1)^2+(a_2)^2} & 0 & 0 \\ 0 & a_1 & -a_2 \\ 0 & a_2 & a_1\end{array}\right)
$$
for smooth complex-valued functions $a_i : N \to \C$ satisfying $(a_1)^2+(a_2)^2\neq 0$. 

We first consider the case where $N$ is one-dimensional. We fix a local coordinate $t$ on $N$. It follows with~\eqref{loctriv} and straightforward calculations that
$$
\left(\varphi \circ (z^1,z^2,\nu\circ p)\right)^*\theta^2_0=\frac{a_1\dot{z}^2-a_2\dot{z}^1}{\left((a_1)^2+(a_2)^2\right)^2}\d t 
$$
where $\dot{z}^i$ denote the derivative of $z^i$ with respect to $t$. Hence we may take 
$$
a_1=\dot{z}^1 \quad \text{and} \quad a_2=\dot{z}^2. 
$$
Writing $\beta=\left(\varphi \circ (z^1,z^2,\nu\circ p)\right)^*\theta^2_1$ and using~\eqref{loctriv} again, we compute
$$
\aligned
\beta=&\left[\dot{z}^1\ddot{z}^2-\dot{z}^2\ddot{z}^1+\left(\dot{z}^1\dot{z}^2(\Pi^2_{21}-\Pi^1_{11})+(\dot{z}^1)^2\Pi^2_{11}-(\dot{z}^2)^2\Pi^1_{12}\right)\dot{z}^1+\right.\\
&\left.+\left(\dot{z}^1\dot{z}^2(\Pi^2_{22}-\Pi^1_{12})+(\dot{z}^1)^2\Pi^2_{12}-(\dot{z}^2)^2\Pi^1_{22}\right)\dot{z}^2\right]\frac{\d t}{(\dot{z}^1)^2+(\dot{z}^2)^2}.
\endaligned
$$
Note that since $\Pi^i_{ik}=0$ for $k=1,2$, it follows that $\beta\equiv 0$ is equivalent to $(z^1,z^2)$ satisfying the following \textsc{ODE} system
$$
\dot{z}^i\left(\ddot{z}^j+\Pi^j_{kl}\dot{z}^k\dot{z}^l\right)=\dot{z}^j\left(\ddot{z}^i+\Pi^i_{kl}\dot{z}^k\dot{z}^l\right), \quad i,j=1,2. 
$$
This last system is easily seen to be equivalent to the system~\eqref{hplanarcurve}. Consequently, the one-dimensional integral manifolds of $E$ are the generalised geodesics of $[\nabla]$. 

Note that in the case of two-dimensional integral manifolds the above computations carry over where $t$ is now a complex parameter, i.e.~the two-dimensional integral manifolds are immersed complex curves $Y \subset M$ for which $\nabla_{\dot{Y}}\dot{Y}$ is proportional to $\dot{Y}$ for some (and hence any) $\nabla \in [\nabla]$. This last condition is equivalent to $Y$ being a totally geodesic immersed complex curve with respect to $([\nabla],J)$ (c.f.~\cite[Lemma 4.1]{MR1348154}) . A totally geodesic immersed complex curve $Y\subset M$ which is maximally extended is called a \textit{complex geodesic}. Since the complex geodesics are the (maximally extended) two-dimensional integral manifolds of $E$, they exist only provided that $E$ is integrable. We will next determine the integrability conditions for $E$. Recall that $E\subset T\mathbb{P}(T^{1,0}M)$ is defined by the equations $\theta^2_1=\theta^2_0=0$ on $B$. It follows with the structure equations \eqref{firststruc} that 
$$
\d \theta^2_0=0 \quad \text{mod}\quad \theta^2_0,\theta^2_1 
$$
and
$$
\d \theta^2_1=W^2_{11\j}\theta^1_0\wedge \ov{\theta^{\jmath}_0}\quad \text{mod}\quad \theta^2_0,\theta^2_1.
$$
Consequently, $E$ is integrable if and only if $W^2_{11\bar 1}=W^2_{11\bar 2}=0$. As a consequence of~\eqref{diffbianchi} and $W^2_{11\bar 1}=0$ we obtain
$$
0=\varphi^2_{11\bar 1}=-W^2_{11\bar 1}\left(\theta^0_0+\ov{\theta^0_0}\right)+W^2_{1s\bar 1}\theta^s_1+W^2_{1s\bar 1}\theta^s_1-W^s_{11\bar 1}\theta^2_s+W^2_{11\bar s}\ov{\theta^s_j},
$$
which is equivalent to $2W^2_{12\bar 1}=W^1_{11\bar 1}$. Using the symmetries of the complex projective Weyl tensor we compute
$$
W^1_{11\bar 1}=-W^2_{21\bar 1}=2W^2_{12\bar 1}=2W^2_{21\bar 1}, 
$$
thus showing that $W^1_{11\bar 1}=W^2_{12\bar 1}=0$. From this we obtain 
$$
0=\varphi^1_{11\bar 1}=2W^1_{12\bar 1}\theta^2_1-W^2_{11\bar 1}\theta^1_2+W^1_{11\bar 2}\ov{\theta^2_1}.
$$
thus implying $W^1_{12\bar 1}=W^2_{11\bar 1}=W^1_{11\bar 2}=0$. Continuing in this vein allows to conclude that all components of the complex projective Weyl tensor must vanish. We may summarise: 
\begin{proposition}
Let $(M,J,[\nabla])$ be a com\-plex pro\-jec\-tive surface. Then the following statements are equivalent:
\begin{itemize}\label{capprop}
\item[(i)] $[\nabla]$ is holomorphic; 
\item[(ii)] The com\-plex pro\-jec\-tive Weyl tensor of $[\nabla]$ vanishes;
\item[(iii)] The rank $2$ bundle $E \to \mathbb{P}(T^{1,0}M)$ is Frobenius integrable; 
\item[(iv)] Every complex line $L\subset T^{1,0}M$ is tangent to a unique complex geodesic. 
\end{itemize}
\end{proposition}
\begin{remark}
The standard flat complex projective structure on $\mathbb{CP}^2$ is holomorphic and the complex geodesics are simply the linearly embedded projective lines $\mathbb{CP}^1\subset \mathbb{CP}^2$.   
\end{remark}

\begin{remark}
Note that the integrability conditions for $E$ are a special case of a more general result obtained by \v{C}ap in~\cite{MR2139714}. There it is shown that $E$ is part of an elliptic CR structure of CR dimension and codimension $2$, which the complex projective structure induces on $\mathbb{P}(T^{1,0}M)$. Furthermore, it is also shown that the integrability of $E$ is equivalent to the holomorphicity of the complex projective surface.  
\end{remark}
 
\section{K\"ahler metrisability}

In this section we will derive necessary conditions for a com\-plex pro\-jec\-tive structure $[\nabla]$ on a complex surface $(M,J)$ to arise via the Levi-Civita connection of a (pseudo-)K\"ahler metric. There exists a complex projectively invariant linear first order differential operator acting on $J$-hermitian $(2,\! 0)$ tensor fields on $M$ with weight $1/3$, i.e~sections of the bundle $S^2_J(TM)\otimes \left(\Lambda^4(T^*M)\right)^{1/3}$. This differential operator has the property that nondegenerate sections in its kernel are in one-to-one correspondence with (pseudo-)K\"ahler metrics on $M$ whose Levi-Civita connection is compatible with $[\nabla]$ (see~\cite{MR492674,MR2998672,MR1619720}).  
\subsection{The differential analysis}

We will show that in the surface case, the (pseudo-)K\"ahler metrics on $(M,J,[\nabla])$ whose Levi-Civita connection is compatible with $[\nabla]$ can equivalently be characterised in terms of a differential system on Cartan's bundle $(\pi : B \to M,\theta)$. 

\begin{proposition}\label{upstairsmetri}
Suppose the (pseudo-)K\"ahler metric $g$ is compatible with $[\nabla]$. Then, writing $\pi^*g=g_{i\j}\theta^i_0\circ\ov{\theta^{\jmath}_0}$ and setting $h_{i\j}=g_{i\j}\left(g_{1\bar 1}g_{2\bar 2}-|g_{1\bar 2}|^2\right)^{-2/3}$,
we have
\begin{equation}\label{difsys}
\d h_{i\bar \jmath}=h_{i\bar\jmath}\left(\theta^0_0+\ov{\theta^0_0}\,\right)+h_{i\bar s}\ov{\theta^s_{j}}+h_{s\bar\jmath}\theta^s_i +h_i\ov{\eps_{sj}\theta^s_0}+\ov{h_{j}}\varepsilon_{si}\theta^s_0
\end{equation}
for some complex-valued functions $h_i$ on $B$. Conversely, suppose there exist com\-plex-valued functions $h_{i\bar \jmath}=\ov{h_{j\bar \imath}}$ and $h_i$ on $B$ solving \eqref{difsys} and satisfying $\left(h_{1\bar 1}h_{2\bar 2}-|h_{1\bar 2}|^2\right)\neq 0$, then the symmetric $2$-form
$$
h_{i\j}\left(h_{1\bar 1}h_{2\bar 2}-|h_{1\bar 2}|^2\right)^{-2}\theta^i_0\circ \ov{\theta^{\jmath}_0}
$$
is the $\pi$-pullback of a $[\nabla]$-compatible (pseudo-)K\"ahler metric on $M$. 
\end{proposition}
\begin{proof}
Let $g$ be a (pseudo-)K\"ahler metric on $(M,J)$ and write $g=g_{i\j}\,\d z^i\circ \d \ov{z^\jmath}$ for local holomorphic coordinates $z=(z^1,z^2) : U \to \C^2$ on $M$. Denoting by $\nabla$ the Levi-Civita connection of $g$, on $U$ the identity $\nabla g=0$ is equivalent to  
$$
\frac{\partial g_{k\j}}{\partial z^i}=g_{s\j}\Gamma^s_{ik}\qquad \text{and}\qquad \frac{\partial g_{k\j}}{\partial \ov{z^{\imath}}}=g_{k\bar s}\ov{\Gamma^s_{ij}},
$$
where $\Gamma^i_{jk}$ denote the complex Christoffel symbols of $\nabla$. Abbreviate $G=\det g_{i\j}$, then we obtain
$$
\frac{\partial G}{\partial z^i}=G\,\Gamma^s_{si}.
$$
Hence, the partial derivative of $h_{k\j}=g_{k\j}\,G^{-2/3}$ with respect to $z^i$ becomes 
$$
\aligned
\frac{\partial h_{k\j}}{\partial z^i}&=g_{l\j}\,\Gamma^l_{ik}\,G^{-2/3}-\frac{2}{3}g_{k\j}\,\Gamma^s_{si}\,G^{-2/3}=h_{l\j}\left(\Gamma^l_{ik}-\frac{2}{3}\Gamma^s_{si}\delta^l_k\right)\\
&=h_{l\j}\left(\Gamma^l_{ik}-\frac{1}{3}\Gamma^s_{si}\delta^l_k-\frac{1}{3}\Gamma^s_{sk}\delta^l_i\right)-\frac{1}{3}h_{l\j}\left(\Gamma^s_{si}\delta^l_k-\Gamma^s_{sk}\delta^l_i\right).
\endaligned
$$
Note that the last two summands in the last equation are antisymmetric in $i,k$, so that we may write
$$
-\frac{1}{3}h_{l\j}\left(\Gamma^s_{si}\delta^l_k-\Gamma^s_{sk}\delta^l_i\right)=\ov{h_j}\eps_{ik}
$$
for unique complex-valued functions $h_i$ on $U$. We thus get
\begin{equation}\label{pullba}
\frac{\partial h_{k\j}}{\partial z^i}=h_{s\j}\Pi^s_{ik}+\ov{h_j}\eps_{ik}. 
\end{equation}
In entirely analogous fashion we obtain
\begin{equation}\label{pullbb}
\frac{\partial h_{k\j}}{\partial \ov{z^{\imath}}}=h_{k\bar s}\ov{\Pi^s_{ij}}+h_k\ov{\eps_{ij}}.
\end{equation}
Recall from Theorem~\ref{cartangeom} that the coordinate system $z : U \to \C^2$ induces a unique section $\sigma_z : U \to B$ of Cartan's bundle such that
\begin{equation}\label{pullbc}
\left(\sigma_z\right)^*\theta^0_0=0, \qquad \left(\sigma_z\right)^*\theta^i_0=\d z^i, \qquad \left(\sigma_z\right)^*\theta^i_j=\Pi^i_{jk}\d z^k. 
\end{equation}
Consequently, using~(\ref{pullba},\,\ref{pullbb},\,\ref{pullbc}) we see that~\eqref{difsys} is necessary. 

Conversely, suppose there exist complex-valued functions $h_{i\j}=\ov{h_{j\bar\imath}}$ and $h_i$ on $B$ solving \eqref{difsys} for which 
$$\left(h_{1\bar 1}h_{2\bar 2}-|h_{1\bar 2}|^2\right)\neq 0.
$$
Setting $g_{i\j}=h_{i\j}\left(h_{1\bar 1}h_{2\bar 2}-|h_{1\bar 2}|^2\right)^{-2}$ we get
\begin{equation}\label{suffdifsys}
\d g_{i\bar \jmath}=-g_{i\bar \jmath}\left(\theta^0_0+\bar\theta^0_0\right)+g_{i\bar s}\ov{\theta^s_{j}}+g_{s\bar\jmath}\theta^s_i+g_{i\bar \jmath\bar s}\ov{\theta^s_0}+g_{i\bar \jmath s}\theta^s_0
\end{equation}
with
$$
g_{i\bar \jmath \bar s}=\frac{(h_{i\bar \jmath}h_{l\bar s}+h_{i\bar s}h_{l\bar \jmath})\eps^{lk}h_k}{(h_{1\bar 1}h_{2\bar 2}-|h_{1\bar 2}|^2)^3}, \quad \text{and} \quad g_{i\bar \jmath k}=\frac{(h_{i\bar \jmath}h_{k\bar s}+h_{k\j}h_{i\bar s})\ov{\eps^{su}{h_u}}}{(h_{1\bar 1}h_{2\bar 2}-|h_{1\bar 2}|^2)^3}.
$$
It follows with~\eqref{suffdifsys} that there exists a unique $J$-Hermitian metric $g$ on $M$ such that $\pi^*g=g_{i\j}\,\theta^i_0\circ \ov{\theta^{\jmath}_0}$. Choose local holomorphic coordinates $z=(z^1,z^2) : U \to \C^2$ on $M$. By abuse of notation we will write $g_{i\j},g_{i\bar \jmath \bar s},g_{i\bar \jmath k}$ for the pullback of the respective functions on $B$ by the section $\sigma_z : U \to B$ associated to $z$. From~\eqref{suffdifsys} we obtain
$$
\frac{\partial g_{i\j}}{\partial z^s}=g_{u\j}\Pi^u_{is}+g_{i\j s}=g_{u\j}\left(\Pi^u_{is}+g^{\bar v u}g_{i\bar v s}\right)=g_{u\j}\left(\Pi^u_{is}+\delta^u_ib_s+\delta^u_sb_i\right)=g_{u\j}\Gamma^u_{is}
$$
where we write
$$
b_i=\frac{h_{i\bar s}\ov{\eps^{su}h_u}}{(h_{1\bar 1}h_{2\bar 2}-|h_{1\bar 2}|^2)^{11/3}}\qquad \text{and} \qquad \Gamma^i_{jk}=\Pi^i_{jk}+\delta^i_jb_k+\delta^i_kb_j.  
$$
Likewise we obtain
$$
\frac{\partial g_{i\j}}{\partial \ov{z^{s}}}=g_{i\bar u}\ov{\Gamma^u_{js}}.
$$
It follows that there exists a complex-linear connection $\nabla$ on $U$ defining $[\nabla]$ and whose complex Christoffel symbols are given by $\Gamma^i_{jk}$. By construction, the connection $\nabla$ preserves $g$ and hence must be the Levi-Civita connection of $g$. Furthermore, $\nabla$ being complex-linear implies that $g$ is K\"ahler. This completes the proof.    
\end{proof}
\subsubsection{First prolongation}
Differentiating \eqref{difsys} yields
\begin{equation}\label{difsys1}
0=\d^2 h_{i\j}=\eps_{si}\ov{\eta_j}\wedge\theta^s_0+\ov{\eps_{sj}}\eta_i\wedge\ov{\theta^s_0}-(h_{s\j}W^{s}_{iv\bar u}+h_{i\bar s}\ov{W^{s}_{ju\bar v}})\ov{\theta^u_0}\wedge\theta^v_0 
\end{equation}
with
$$
\eta_k=\d h_k+h_k\left(\ov{\theta^0_0}-\theta^0_0\right)-h_j\theta^j_k+\ov{\eps^{ij}}h_{k\j}\ov{\theta^0_i}.
$$
This implies that we can write
\begin{equation}\label{ansatz}
\eta_i=a_{ij}\theta^j_0
\end{equation}
for unique complex-valued functions $a_{ij}$ on $B$. Equations \eqref{difsys1} and \eqref{ansatz} imply
\begin{equation}\label{algident}
\eps_{ki}\ov{a_{jl}}-\ov{\eps_{lj}} a_{ik}=\ov{h_{j\bar s}}W^s_{ik\bar l}-h_{i\bar s}\ov{W^s_{jl\bar k}}
\end{equation}
Contracting this last equation with $\ov{\eps^{jl}}\eps^{ik}$ implies that the function 
$$
h=-\frac{1}{2}\ov{\eps^{ij}a_{ij}}
$$
is real-valued. We get
$$
a_{jl}=\eps_{jl}h-\frac{1}{2}\ov{\eps^{iu}}h_{s\bar \imath}W^s_{jl\bar u}, 
$$
and thus
$$
\d h_i=h_i\left(\theta^0_0-\ov{\theta^0_0}\,\right)+h_j\theta^j_i+h_{i\bar s}\ov{\eps^{sl}\theta^0_l}+\left(\eps_{ij}h-\frac{1}{2}\ov{\eps^{uv}} h_{s\bar u}W^s_{ij\bar v}\right)\theta^j_0. 
$$
Plugging the formula for $a_{ij}$ back into \eqref{algident} yields the integrability conditions
$$
h_{s\j}W^s_{ik\bar l}-h_{i\bar s}\ov{W^s_{jl\bar k}}=\frac{1}{2}\ov{\eps_{lj}\eps^{uv}}h_{s\bar u}W^s_{ik\bar v}-\frac{1}{2}\eps_{ki}\eps^{uv}h_{u\bar s}\ov{W^s_{jl\bar v}}.
$$
This last equation can be simplified so that we obtain: 
\begin{proposition}
A necessary condition for a com\-plex pro\-jec\-tive surface $(M,J,[\nabla])$ to be K\"ahler metrisable is that 
\begin{equation}\label{firstinte}
\ov{h_{j\bar s}}W^s_{ik\bar l}+\ov{h_{l\bar s}}W^s_{ik\j}=h_{k\bar s}\ov{W^s_{jl\bar \imath}}+h_{i\bar s}\ov{W^s_{jl\bar k}}
\end{equation}
admits a nondegenerate solution $\ov{h_{i\j}}=h_{j\bar\imath}$.    
\end{proposition}
\begin{remark}
Note that under suitable constant rank assumptions the system \eqref{firstinte} defines a subbundle of the bundle over $M$ whose sections are hermitian forms on $(M,J)$. For a generic complex projective structure $[\nabla]$ this subbundle does have rank $0$. 
\end{remark}
\subsubsection{Second prolongation}
We start by computing
$$
0=\d ^2 h_i\wedge\theta^1_0\wedge\theta^2_0=-\left(h_{i\j}\ov{\eps^{jk}L_{k}}\right)\theta^1_0\wedge\ov{\theta^1_0}\wedge\theta^2_0\wedge\ov{\theta^2_0}
$$
which is equivalent to
$$
\left(\begin{array}{cc} h_{1\bar 1} & h_{1\bar 2} \\ h_{2\bar 1} & h_{2\bar 2}\end{array}\right)\cdot \left(\begin{array}{r} \ov{L_2} \\ -\ov{L_1}\end{array}\right)=0
$$
which cannot have any solution with $(h_{11}h_{22}-|h_{12}|^2)\neq 0$ unless $L_1=L_2=0$. This shows: 
\begin{theorem}\label{liouvflat}
A necessary condition for a com\-plex pro\-jec\-tive surface to be K\"ahler metrisable is that it is Liouville-flat, i.e.~its com\-plex pro\-jec\-tive Liouville curvature vanishes.    
\end{theorem}
\begin{remark}
Note that the vanishing of the Liouville curvature is equivalent to requesting that the curvature of $\theta$ is of type $(1,\! 1)$ only, which agrees with general results in~\cite{MR2532439}. 
\end{remark}
Assuming henceforth $L_1=L_2=0$ we also get
\begin{equation}\label{sameold}
0=\d^2h_i=\left(\eps_{ij}\eta+\varphi_{ij}\right)\wedge\theta^j_0
\end{equation}
with
$$
\eta=\d h +2h\Re(\theta^0_0)+2\eps^{ij}\Re(h_i\theta^0_j)-\frac{1}{2}\eps^{kl}h_{k\bar \imath}\ov{\eps^{ij}K_{js\bar l}\theta^s_0}
$$
and
$$
\varphi_{ij}=\d r_{ij}+r_{ij}\ov{\theta^0_0}-r_{si}\theta^s_j-r_{sj}\theta^s_i-h_lW^l_{ij\bar s}\ov{\theta^s_0}+\frac{1}{2}\ov{\eps^{uv}}\left(h_{i\bar u}\ov{K_{vs\j}}+h_{j\bar u}\ov{K_{vs\bar \imath}}\right)\ov{\theta^s_0} 
$$
where
$$
r_{ij}=-\frac{1}{2}\ov{\eps^{uv}} h_{s\bar u}W^s_{ij\bar v}.
$$
It follows with Cartan's lemma that there are functions $a_{ijk}=a_{ikj}$ such that
$$
\eps_{ij}\eta+\varphi_{ij}=a_{ijk}\theta^k_0.
$$
Since $\varphi_{ij}$ is symmetric in $i,j$, this implies
$$
\eta=\frac{1}{2}\eps^{ji}a_{ijs}\theta^s_0. 
$$
Since $h$ is real-valued, we must have
$$
\eps^{ji}a_{ijs}=\ov{\eps^{uv}}\eps^{kl}h_{k\bar u}K_{ls\bar v}. 
$$
Concluding, we get
$$
\d h=-2h\Re(\theta^0_0)+2\eps^{kl}\Re(h_l\theta^0_k)+\frac{1}{2}\ov{\eps^{ij}}\eps^{kl}\Re(h_{k\bar \imath}K_{ls\j}\theta^s_0).
$$
This completes the prolongation procedure.
\begin{remark}
Note that further integrability conditions can be derived from \eqref{sameold}, we won't write these out though. 
\end{remark}
Using Proposition \ref{upstairsmetri} we obtain:
\begin{theorem}\label{mainshit}
Let $(M,J,[\nabla])$ be a com\-plex pro\-jec\-tive surface with Cartan geometry $(\pi : B \to M,\theta)$. If $U\subset B$ is a connected open set on which there exist functions $h_{i\j}=\ov{h_{j\bar\imath}}$, $h_i$ and $h$ that satisfy the rank $9$ linear system
\begin{equation}\label{difsyscomplete}
\aligned
\d h_{i\j}&=2h_{i\j}\Re(\theta^0_0)+h_{i\bar s}\ov{\theta^{s}_j}+h_{s\j}\theta^s_i+h_i\ov{\eps_{sj}\theta^{s}_0}+\ov{h_j}\eps_{si}\theta^s_0,\\
\d h_k&=2\i h_k\Im(\theta^0_0)+h_l\theta^l_k+h_{k\bar \imath}\ov{\eps^{ij}\theta^0_j}+\left(\eps_{kl}h-\frac{1}{2}\ov{\eps^{ij}} h_{s\bar \imath} W^s_{kl\j}\right)\theta^l_0,\\
\d h&=-2h\Re(\theta^0_0)-2\eps^{lk}\Re(h_l\theta^0_k)+\frac{1}{2}\ov{\eps^{ij}}\eps^{kl}\Re(h_{k\bar \imath}K_{ls\j}\theta^s_0),
\endaligned
\end{equation}
and $(h_{1\bar 1}h_{2 \bar 2}-|h_{1\bar 2}|^2)\neq 0$, then the quadratic form
$$
g=\frac{h_{i\j}\theta^i_0\circ \ov{\theta^{\jmath}_0}}{(h_{1\bar 1}h_{2\bar 2}-|h_{1\bar 2}|^2)^2}
$$
is the $\pi$-pullback to $U$ of a (pseudo-)K\"ahler metric on $\pi(U) \subset M$ that is compatible with $[\nabla]$. 
\end{theorem}
From this we get:
\begin{corollary}\label{solflat}
The K\"ahler metrics defined on some domain $U\subset \mathbb{CP}^2$ which are compatible with the standard com\-plex pro\-jec\-tive structure on $\mathbb{CP}^2$ are in one-to-one correspondence with the hermitian forms on $\C^3$ whose rank is at least two. 
\end{corollary}
\begin{proof} Suppose the com\-plex pro\-jec\-tive structure $[\nabla]$ has vanishing com\-plex pro\-jec\-tive Weyl and Liouville curvature. Then the differential system \eqref{difsyscomplete} may be written as 
\begin{equation}\label{completeinte}
\d H+\theta H+H\theta^*=0
\end{equation}
with
$$
H=H^*=\left(\begin{array}{ccc}h & -\ov{h_2} & \ov{h_1}\\ -h_2 & -h_{22} & h_{21}\\ h_1 & h_{12} & -h_{11}\end{array}\right)
$$
where ${}^*$ denotes the conjugate transpose matrix. Recall that in the flat case $\theta=g^{-1}\d g$ for some smooth $g: B \to \mathrm{PSL}(3,\C)$, hence the solutions to \eqref{completeinte} are
$$
H=g^{-1}C\left(g^{-1}\right)^*
$$
where $C=C^*$ is a constant hermitian matrix of rank at least two. The statement now follows immediately with Theorem \ref{mainshit}. 
\end{proof}

\begin{remark}
On can deduce from Corollary \ref{solflat} that a K\"ahler metric $g$ giving rise to flat complex projective structures must have constant holomorphic sectional curvature. A result first proved in~\cite{MR0087181} (in all dimensions). 
\end{remark}

\begin{remark}
One can also ask for existence of complex projective structures $[\nabla]$ whose \textit{degree of mobility} is greater than one, i.e.~they admit several (non-proportional) compatible K\"ahler metrics. In~\cite{MR2948791} (see also~\cite{MR2729086}) it was shown that the only closed complex projective manifold with degree of mobility greater than two is $\mathbb{CP}^n$ with the projective structure arising via the Fubini-Study metric.\end{remark}

\subsubsection*{Acknowledgements}
The author is grateful to V. S. Matveev and S. Rosemann for introducing him to the subject of com\-plex pro\-jec\-tive geometry through many stimulating discussions, some of which took place during very enjoyable visits to Friedrich-Schiller-Universit\"at in Jena. The author also would like to thank R. L. Bryant for sharing with him his notes~\cite{bryanthomopro} which contain the proofs of the counterparts for real projective surfaces of Theorem \ref{mainshit} and Corollary \ref{solflat}. Furthermore, the author also would like to thank the referees for several valuable comments. 

\providecommand{\bysame}{\leavevmode\hbox to3em{\hrulefill}\thinspace}
\providecommand{\noopsort}[1]{}
\providecommand{\mr}[1]{\href{http://www.ams.org/mathscinet-getitem?mr=#1}{MR~#1}}
\providecommand{\zbl}[1]{\href{http://www.zentralblatt-math.org/zmath/en/search/?q=an:#1}{Zbl~#1}}
\providecommand{\jfm}[1]{\href{http://www.emis.de/cgi-bin/JFM-item?#1}{JFM~#1}}
\providecommand{\arxiv}[1]{\href{http://www.arxiv.org/abs/#1}{arXiv~#1}}
\providecommand{\doi}[1]{\href{http://dx.doi.org/#1}{DOI}}
\providecommand{\MR}{\relax\ifhmode\unskip\space\fi MR }
\providecommand{\MRhref}[2]{%
  \href{http://www.ams.org/mathscinet-getitem?mr=#1}{#2}
}
\providecommand{\href}[2]{#2}


\begin{thebibliography}{10}

\bibitem{MR2228318}
\bgroup\scshape{}V.~Apostolov\egroup{}, \bgroup\scshape{}D.~M.~J.
  Calderbank\egroup{}, and \bgroup\scshape{}P.~Gauduchon\egroup{}, Hamiltonian
  2-forms in {K}\"ahler geometry. {I}. {G}eneral theory,  \emph{J. Differential
  Geom.} \textbf{73} (2006), 359--412. \mr{2228318}\;

\bibitem{arxiv:1311.0517}
\bgroup\scshape{}A.~Bolsinov\egroup{}, \bgroup\scshape{}V.~S. Matveev\egroup{},
  \bgroup\scshape{}T.~Mettler\egroup{}, and
  \bgroup\scshape{}S.~Rosemann\egroup{}, Four-dimensional {K}\"ahler metrics
  admitting c-projective vector fields, (2013). \arxiv{1311.0517}.

\bibitem{MR2581355}
\bgroup\scshape{}R.~Bryant\egroup{}, \bgroup\scshape{}M.~Dunajski\egroup{}, and
  \bgroup\scshape{}M.~Eastwood\egroup{}, Metrisability of two-dimensional
  projective structures,  \emph{J. Differential Geom.} \textbf{83} (2009),
  465--499. \mr{2581355}\;

\bibitem{bryanthomopro}
\bgroup\scshape{}R.~L. Bryant\egroup{}, \emph{Notes on projective surfaces},
  private manuscript in progress.

\bibitem{MR1656822}
\bgroup\scshape{}D.~M.~J. Calderbank\egroup{}, M\"obius structures and
  two-dimensional {E}instein-{W}eyl geometry,  \emph{J. Reine Angew. Math.}
  \textbf{504} (1998), 37--53. \mr{1656822}\;

\bibitem{MR1856258}
\bgroup\scshape{}D.~M.~J. Calderbank\egroup{} and
  \bgroup\scshape{}T.~Diemer\egroup{}, Differential invariants and curved
  {B}ernstein-{G}elfand-{G}elfand sequences,  \emph{J. Reine Angew. Math.}
  \textbf{537} (2001), 67--103. \mr{1856258}\;

\bibitem{MR3158041}
\bgroup\scshape{}A.~{\v{C}}ap\egroup{}, \bgroup\scshape{}A.~R. Gover\egroup{},
  and \bgroup\scshape{}H.~R. Macbeth\egroup{}, Einstein metrics in projective
  geometry,  \emph{Geom. Dedicata} \textbf{168} (2014), 235--244.
  \mr{3158041}\;

\bibitem{MR2139714}
\bgroup\scshape{}A.~{\v{C}}ap\egroup{}, Correspondence spaces and twistor
  spaces for parabolic geometries,  \emph{J. Reine Angew. Math.} \textbf{582}
  (2005), 143--172. \mr{2139714}\;

\bibitem{MR2532439}
\bgroup\scshape{}A.~{\v{C}}ap\egroup{} and
  \bgroup\scshape{}J.~Slov{\'a}k\egroup{}, \emph{Parabolic geometries. {I}},
  \emph{Mathematical Surveys and Monographs} \textbf{154}, American
  Mathematical Society, Providence, RI, 2009, Background and general theory.
  \mr{2532439}\;

\bibitem{MR1847589}
\bgroup\scshape{}A.~{\v{C}}ap\egroup{}, \bgroup\scshape{}J.~Slov\'ak\egroup{},
  and \bgroup\scshape{}V.~Sou\v{c}ek\egroup{}, Bernstein-{G}elfand-{G}elfand
  sequences,  \emph{Ann. of Math. (2)} \textbf{154} (2001), 97--113.
  \mr{1847589}\;

\bibitem{MR1504846}
\bgroup\scshape{}E.~Cartan\egroup{}, Sur les vari\'et\'es \`a connexion
  projective,  \emph{Bull. Soc. Math. France} \textbf{52} (1924), 205--241.
  \mr{1504846}\;

\bibitem{MR492674}
\bgroup\scshape{}V.~V. Doma{\v{s}}ev\egroup{} and
  \bgroup\scshape{}{\u{I}}.~Mike{\v{s}}\egroup{}, On the theory of
  holomorphically projective mappings of {K}\"ahlerian spaces,  \emph{Mat.
  Zametki} \textbf{23} (1978), 297--303. \mr{492674}\;

\bibitem{MR2574978}
\bgroup\scshape{}S.~Dumitrescu\egroup{}, Connexions affines et projectives sur
  les surfaces complexes compactes,  \emph{Math. Z.} \textbf{264} (2010),
  301--316. \mr{2574978}\;

\bibitem{MR2384718}
\bgroup\scshape{}M.~Eastwood\egroup{} and \bgroup\scshape{}V.~Matveev\egroup{},
  Metric connections in projective differential geometry,  in \emph{Symmetries
  and overdetermined systems of partial differential equations}, \emph{IMA Vol.
  Math. Appl.} \textbf{144}, Springer, New York, 2008, pp.~339--350.
  \mr{2384718}\;

\bibitem{MR2948791}
\bgroup\scshape{}A.~Fedorova\egroup{}, \bgroup\scshape{}V.~Kiosak\egroup{},
  \bgroup\scshape{}V.~S. Matveev\egroup{}, and
  \bgroup\scshape{}S.~Rosemann\egroup{}, The only {K}\"ahler manifold with
  degree of mobility at least 3 is {$(\Bbb CP(n),g_{\text{Fubini-Study}})$},
  \emph{Proc. Lond. Math. Soc. (3)} \textbf{105} (2012), 153--188.
  \mr{2948791}\;

\bibitem{MR3159950}
\bgroup\scshape{}A.~R. Gover\egroup{} and \bgroup\scshape{}H.~R.
  Macbeth\egroup{}, Detecting {E}instein geodesics: {E}instein metrics in
  projective and conformal geometry,  \emph{Differential Geom. Appl.}
  \textbf{33} (2014), 44--69. \mr{3159950}\;

\bibitem{MR505691}
\bgroup\scshape{}R.~C. Gunning\egroup{}, \emph{On uniformization of complex
  manifolds: the role of connections}, \emph{Mathematical Notes} \textbf{22},
  Princeton University Press, Princeton, N.J., 1978. \mr{505691}\;

\bibitem{MR2984590}
\bgroup\scshape{}M.~Hammerl\egroup{}, \bgroup\scshape{}P.~Somberg\egroup{},
  \bgroup\scshape{}V.~Sou{\v{c}}ek\egroup{}, and
  \bgroup\scshape{}J.~{\v{S}}ilhan\egroup{}, On a new normalization for tractor
  covariant derivatives,  \emph{J. Eur. Math. Soc.} \textbf{14} (2012),
  1859--1883. \mr{2984590}\;

\bibitem{MR699802}
\bgroup\scshape{}N.~J. Hitchin\egroup{}, Complex manifolds and {E}instein's
  equations,  in \emph{Twistor geometry and nonlinear systems ({P}rimorsko,
  1980)}, \emph{Lecture Notes in Math.} \textbf{970}, Springer, Berlin, 1982,
  pp.~73--99. \mr{699802}\;

\bibitem{MR2591680}
\bgroup\scshape{}J.~Hrdina\egroup{}, Almost complex projective structures and
  their morphisms,  \emph{Arch. Math. (Brno)} \textbf{45} (2009), 255--264.
  \mr{2591680}\;

\bibitem{MR2729086}
\bgroup\scshape{}K.~Kiyohara\egroup{} and \bgroup\scshape{}P.~Topalov\egroup{},
  On {L}iouville integrability of {$h$}-projectively equivalent {K}\"ahler
  metrics,  \emph{Proc. Amer. Math. Soc.} \textbf{139} (2011), 231--242.
  \mr{2729086}\;

\bibitem{MR575449}
\bgroup\scshape{}S.~Kobayashi\egroup{} and \bgroup\scshape{}T.~Ochiai\egroup{},
  Holomorphic projective structures on compact complex surfaces,  \emph{Math.
  Ann.} \textbf{249} (1980), 75--94. \mr{575449}\;

\bibitem{arXiv:1303.4912}
\bgroup\scshape{}W.~Kry\'nski\egroup{}, Webs and projective structures on a
  plane, (2013). \arxiv{1303.4912}.

\bibitem{19.0317.02}
\bgroup\scshape{}R.~Liouville\egroup{}, {Sur une classe d'\'equations
  diff\'erentielles, parmi lesquelles, en particulier, toutes celles des lignes
  g\'eod\'esiques se trouvent comprises.},  \emph{Comptes rendus hebdomadaires
  des s\'eances de l'{A}cad\'emie des sciences} \textbf{105} (1887),
  1062--1064.

\bibitem{21.0317.01}
\bgroup\scshape{}R.~Liouville\egroup{}, Sur les invariants de certaines
  \'equations diff\'erentielles et sur leurs applications,  \emph{Journal de
  l'Ecole Polytechnique} \textbf{59} (1889), 7--76.

\bibitem{MR2998672}
\bgroup\scshape{}V.~S. Matveev\egroup{} and
  \bgroup\scshape{}S.~Rosemann\egroup{}, Proof of the {Y}ano-{O}bata conjecture
  for h-projective transformations,  \emph{J. Differential Geom.} \textbf{92}
  (2012), 221--261. \mr{2998672}\;

\bibitem{MR2770434}
\bgroup\scshape{}B.~McKay\egroup{}, Characteristic forms of complex {C}artan
  geometries,  \emph{Adv. Geom.} \textbf{11} (2011), 139--168. \mr{2770434}\;

\bibitem{MR3043749}
\bgroup\scshape{}T.~Mettler\egroup{}, Reduction of {$\beta$}-integrable
  2-{S}egre structures,  \emph{Comm. Anal. Geom.} \textbf{21} (2013), 331--353.
  \mr{3043749}\;

\bibitem{MR3144212}
\bysame, Weyl metrisability of two-dimensional projective structures,
  \emph{Math. Proc. Cambridge Philos. Soc.} \textbf{156} (2014), 99--113.
  \mr{3144212}\;

\bibitem{MR1384327}
\bgroup\scshape{}J.~Mike{\v{s}}\egroup{}, Geodesic mappings of affine-connected
  and {R}iemannian spaces,  \emph{J. Math. Sci.} \textbf{78} (1996), 311--333,
  Geometry, 2. \mr{1384327}\;

\bibitem{MR1619720}
\bysame, Holomorphically projective mappings and their generalizations,
  \emph{J. Math. Sci. (New York)} \textbf{89} (1998), 1334--1353, Geometry, 3.
  \mr{1619720}\;

\bibitem{MR1348154}
\bgroup\scshape{}R.~Molzon\egroup{} and \bgroup\scshape{}K.~P.
  Mortensen\egroup{}, The {S}chwarzian derivative for maps between manifolds
  with complex projective connections,  \emph{Trans. Amer. Math. Soc.}
  \textbf{348} (1996), 3015--3036. \mr{1348154}\;

\bibitem{MR2876789}
\bgroup\scshape{}P.~Nurowski\egroup{}, Projective versus metric structures,
  \emph{J. Geom. Phys.} \textbf{62} (2012), 657--674. \mr{2876789}\;

\bibitem{MR0284936}
\bgroup\scshape{}T.~Ochiai\egroup{}, Geometry associated with semisimple flat
  homogeneous spaces,  \emph{Trans. Amer. Math. Soc.} \textbf{152} (1970),
  159--193. \mr{0284936}\;

\bibitem{MR0066024}
\bgroup\scshape{}T.~{\=O}tsuki\egroup{} and
  \bgroup\scshape{}Y.~Tashiro\egroup{}, On curves in {K}aehlerian spaces,
  \emph{Math. J. Okayama Univ.} \textbf{4} (1954), 57--78. \mr{0066024}\;

\bibitem{MR3144367}
\bgroup\scshape{}M.~Randall\egroup{}, Local obstructions to projective surfaces
  admitting skew-symmetric {R}icci tensor,  \emph{J. Geom. Phys.} \textbf{76}
  (2014), 192--199. \mr{3144367}\;

\bibitem{MR0087181}
\bgroup\scshape{}Y.~Tashiro\egroup{}, On a holomorphically projective
  correspondence in an almost complex space,  \emph{Math. J. Okayama Univ.}
  \textbf{6} (1957), 147--152. \mr{0087181}\;

\bibitem{51.0569.03}
\bgroup\scshape{}T.~Y. Thomas\egroup{}, {On the projective and equi-projective
  geometries of paths.},  \emph{Proc. Nat. Acad. Sci.} \textbf{11} (1925),
  199--203.

\bibitem{48.0844.04}
\bgroup\scshape{}H.~Weyl\egroup{}, {Zur Infinitesimalgeometrie: Einordnung der
  projektiven und der konformen Auffassung.},  \emph{G\"ottingen Nachrichten}
  (1921), 99--112.

\bibitem{MR0500679}
\bgroup\scshape{}Y.~Yoshimatsu\egroup{}, {$H$}-projective connections and
  {$H$}-projective transformations,  \emph{Osaka J. Math.} \textbf{15} (1978),
  435--459. \mr{0500679}\;

\end{thebibliography}
\end{document}